\documentclass[12pt]{article}
\usepackage[backend=bibtex, style=numeric]{biblatex}
\usepackage{showlabels}
\usepackage{amsmath, amsthm, amssymb,amscd, bbm}
\usepackage{fullpage, paralist,enumerate}
\usepackage{verbatim}
\usepackage[all]{xy}
\usepackage[parfill]{parskip}
\usepackage{graphicx}
\usepackage{subcaption}
\usepackage{mathabx}
\usepackage{float}
\CompileMatrices

\usepackage[shortlabels] {enumitem}
\usepackage{xcolor}
\usepackage{hyperref}
\hypersetup{
    colorlinks   = true,
    citecolor    = blue,
    urlcolor=blue,
}

\newcommand{\R}{{\mathbb R}}

\newcommand{\abs}[1]{\left| #1 \right|}
\renewcommand{\vec}[1]{\mathbf{#1}}

\newtheorem{theorem}{Theorem}[section]
\newtheorem{lemma}[theorem]{Lemma}
\newtheorem{proposition}[theorem]{Proposition}
\newtheorem{corollary}[theorem]{Corollary}

\theoremstyle{remark}

\theoremstyle{remark}

\newcommand{\norm}[1]{\left|\left|#1\right|\right|}

\newcommand{\cl}{{\mathrm{cl}\,}}

\usepackage{authblk}

\title{Log-Concave Duality in Estimation and Control\\}

\vspace{1cm}

\author{Robert Bassett\thanks{rbassett@math.ucdavis.edu}}
\author{Michael Casey\thanks{msc.sedcontra@gmail.com}}
\author{Roger J-B Wets\thanks{rjbwets@ucdavis.edu}}
\affil{Department of Mathematics, Univ. California Davis}

\date{}
\begin{document}
\maketitle

\begin{abstract} In this paper we generalize the estimation-control
duality that exists in the linear-quadratic-Gaussian setting.
We extend this duality to maximum a posteriori estimation of the
system's state, where the measurement and
dynamical system noise are independent log-concave random variables.
More generally, we show that a problem which induces 
a convex penalty on noise terms will have a dual control problem.
We provide conditions for strong duality to hold, and then prove
relaxed conditions for the piecewise linear-quadratic case.
The results have applications in estimation problems with nonsmooth
densities, such as log-concave maximum likelihood densities.
 We conclude with an example reconstructing optimal
estimates from solutions to the dual control problem, which has
implications for sharing solution methods between the two types of
problems. 
\end{abstract}

\section{Introduction}
We consider the problem of estimating the state of a noisy dynamical
system  based only on noisy measurements of the system. In this paper,
we assume linear dynamics, so that the progression of state
variables is
\begin{align} \label{estdynamics}
\vec{X}_{t+1} =& F_{t}(\vec{X}_{t}) + \vec{W}_{t+1}, \; \; \;
t=0,...,T-1 \\
\vec{X}_{0} =& \vec{W}_{0}
\end{align}
$\vec{X}_{t}$ is the state variable--a random vector indexed by a discrete time-step
$t$ which ranges from $0$ to some final time $T$. 
All of the results in this paper still hold in the case that the 
dimension of $\vec{X}_{t}$
is time-dependent, but for notational convenience we will assume that
$\vec{X}_{t} \in \mathbb{R}^{n_{x}}$ for $t=0,...,T$. $F_{t}$ is then a
$n_{x} \times n_{x}$ real-valued matrix that, though it may vary with
time, is known a priori. The $\vec{W}_{t}$ term is a random vector in $\mathbb{R}^{n_x}$ that
represents noise in the system dynamics. Note that, in this
formulation, the random vectors $\vec{W}_{t}$ are \emph{primitive}, in
the sense that they generate all of the randomness associated with the
problem. The state variables $\vec{X}_{t}$ are secondary, being
derived from applying dynamic equations to $\vec{W}_{t}$ terms.

In addition to the dynamics that govern the state progression, we also
have a measurement process which dictates the observable information
at time $t$. We assume that the measurement process is linear.
\begin{equation} \label{estmeasure}
\vec{Z}_{t} = H_{t}(\vec{X}_{t}) + \vec{V}_{t}
\end{equation}
The vector $\vec{Z}_{t}$ is a (secondary) random vector of dimension $n_{z}$. 
Again, we can consider the case that the dimension of $\vec{Z}_{t}$ changes with
time, but for notational convenience we will assume that the
measurements have a fixed dimension. $H_{t}$ is than an $n_{z} \times
n_{x}$ matrix, which similar to $F_{t}$ may vary with time but is
known in advance. $\vec{V}_{t}$ is a primitive random vector of
dimension $\mathbb{R}^{n_{z}}$ that represents measurement noise. 

Different information structures in this setup correspond to different
types of estimation problem. In this paper we consider the smoothing
problem, which consists of estimating of
$\vec{X_{0}},...,\vec{X}_{T}$
after all measurement variables $\vec{Z}_{0},...\vec{Z}_{T}$ have been
observed. In this sense, the information associated with the problem
is constant--the set of measurements which we use to
estimate $\vec{X}_{0}$ is the same as the measurements with which we
estimate $\vec{X}_{T}$ This differs from the filtering problem,
which is one of sequential state estimation. In the filtering problem,
the set of available measurements depends on the time of the state
being estimated. Of course, the difference between these problems 
can be formulated in terms of measurability with respect to certain
filtrations, but we avoid this language because our main problem of
interest will end be deterministic.

Kalman, in his seminal paper \cite{Kalman},
assumed that there was no measurement noise associated with system, so
that $\vec{V}_{t} \equiv 0$. Motivated by minimizing mean-squared
error, he sought to find the conditional expectation of the states
given the measurement.
Under the assumption that the dynamic and measurement noise are
Gaussian, conditional expectation reduces to a deterministic maximum a posteriori
(MAP) problem, in which the optimal estimate is the mode of the
conditional density
$$f_{X}(\vec{X}_{0},...,\vec{X}_{T}|(\vec{Z}_{0},...,\vec{Z}_{T})=(z_{0},...,z_{T}))$$

Assuming that $\vec{W}_{t} \sim \mathcal{N}(0,Q_{t})$ and
$\vec{V}_{t} \equiv 0$, the problem can be derived explicitly, as in
\cite{Cox}. We are left with the following optimization problem:

\begin{align*} \label{prob:KEst} \tag{$\mathcal{P}_{Kal}$} 
\min_{x_{0},...,x_{T}} & \sum_{t=0}^{T}  \frac{1}{2} w_{t}' Q_{t}^{-1}
w_{t} \\
\text{subject to } & x_{t+1} = F_{t} x_{t} + w_{t+1}, \; \;
t=0,...,T-1\\
&x_{0} = w_{0} \\
& z_{t} = H_{t} x_{t}, \; \; t=0,...,T
\end{align*}

In this formulation, the variables $x_{t}$ and $w_{t}$ are estimates
of the random variables $\vec{X}_{t}$ and $\vec{W}_{t}$, respectively.
In this sense, $w_{t}$ is also a decision variable, but we prefer to
write the problem in this reduced formulation where a decision
$x_{0},...,x_{T}$ generates the variables $w_{0},...,w_{T}$.

Kalman observed that the Linear Quadratic Regulator problem
\begin{align*} \label{prob:KLQR} \tag{$D_{Kal}$}
\min_{u_{0},...,u_{T}} &\sum_{t=0}^{T}  \frac{1}{2} y_{t}' Q_{t} y_{t} \\ 
\text{subject to } & y_{t+1} = F_{t}' y_{t} + H_{t}' u_{t+1} \\
& y_{0} = H_{0}' u_{0}
\end{align*}
over controls $\{u_{t}\}_{t=0}^{T}$ and states $\{y_{t}\}_{t=0}^{T}$,
is \emph{dual} to the estimation problem above. Kalman defined this
duality in terms of the equations that characterize their solutions:
the algebraic Riccati equation which characterizes the value function of
\eqref{prob:KLQR} is the same equation that governs the propagation of
the variance of the estimate in \eqref{prob:KEst}, with a
time reversal. Since the Linear Quadratic Regulator problem is one of
optimal control, the relationship between the problems is
described as duality between estimation and control, in the
Linear-Quadratic Gaussian setting.

Compared to the equation-correspondence duality which is typical in
the engineering literature \cite{Todorov},
\cite{LinEstControl} , we take a different approach by 
using the duality theory of convex programming. The allows us to
extend the duality of estimation and control to the more general
setting where noise and measurement noise have \emph{log-concave}
densities. This includes the Linear-Quadratic Gaussian 
framework, but by viewing the duality in a convex-analytic
framework we gain more insight into the relationship between
estimation and control. Previous literature has focused almost
exclusively on either equation-correspondence duality or convex
analytic duality between estimation and control problems. We will
investigate the relationship 
between these two notions of duality in a future paper, but for now
consider the convex-analytic case.



The rest of this paper is organized as follows. In section 2 we state
and prove the main result of the paper: a duality result between
estimation and optimal control when the noise terms in \eqref{estmeasure},
\eqref{estdynamics} are
log-concave. Section 3 applies this result to the case where the
noise terms have densities which are exponentiated monitoring
functions, so that no constraint qualification is required for strong
duality. Section 4 contains a practical example, where the solution to
the optimal control problem is used to generate an optimal state estimate.


We conclude this section by establishing some definitions and notations that we will use
throughout the rest of the paper. Recall that a function is called
\emph{lower semi-continuous} (lsc) if for every $x$ in its domain,
$$\liminf_{x^{\nu} \to x} f(x^{\nu}) \geq f(x)$$
for every sequence $x^{\nu} \to x$. In addition, recall that a convex function
which takes extended real-values is called \emph{proper} if it is not
identically $\infty$ and never takes the value $- \infty$
\cite{VarAnal}.  In keeping with convex-analytic literature, we refer
to the domain of an extended-valued convex function as the set where
it assumes a finite value.

Lastly, given a convex function $f: \R^n \to (-\infty, \infty]$,
 we denote by $f^{*}$ the convex conjugate of $f$, which is defined as
$$f^{*}(y) = \sup_{x \in \R^n} \left\{ \langle x, y \rangle - f(x)
\right\}.$$
Conjugation is ubiquitous in convex-analytic duality theory, and this
paper is no exception. For more details and background on conjugation
and its relation to duality, the reader
should consult any of \cite{ConjDual} \cite{ConvAnal} \cite{VarAnal}.



\section{Estimation with Convex Penalties}

In this section we consider the case where the 
 the random vectors $\vec{W}_{t}$ and $\vec{V}_{t}$ in
\eqref{estdynamics} and \eqref{estmeasure} 
have log-concave density functions. Recall that a function $\phi:\R^d \to
\R$ is \emph{log-concave} if 
$$ \phi(x) = \exp -f(x)$$
where $f: \R^d \to (-\infty, \infty]$ is a convex function.
By convention, we adopt that $e^{-\infty} =0$.

The collection of random vectors with log-concave densities is broad
 enough to include
many commonly used distributions, such as the normal,
Laplace, and exponential \cite{LogConcaveApp}. Moreover, it
is closed with respect to taking marginals, convolutions, and forming
product measures \cite{UniConvApp}. 
These characteristics make MAP estimation in the presence of log-concave 
noise much more amenable to computation than the more general unimodal
class, because they guarantee that conditional
expectations, sums of random variables, and joint densities formed by
independent log-concave random variables remain log-concave. These are
exactly the operations performed when considering MAP estimation in the presence
of linear dynamics and measurements. The broader class of unimodal
distributions, on the other hand, does not enjoy these properties,
making them much more difficult to work with in the context of
discrete-time state estimation.

Nonparametric density estimation
within the class of log-concave random vectors also has attractive theoretical and computational
properties. We will not review those results here (see, for example,
\cite{CRANPackage, pal, rufibach}), but we do comment that the results in the later
sections find rich applications in nonparametric log-concave density
estimation, particularly because of their non-smooth nature. 

The maximum a posteriori estimation of the states,
given measurements $z_{0},...,z_{T}$ can be derived similarly to the
Gaussian case.
\begin{proposition}
Assume that $\vec{W}_{t}$ and $\vec{V}_{t}$ are independent and
have log-concave densities $e^{-f_{t}}$ and $e^{-g_{t}}$,
respectively, for $t=0,...,T$. Then the maximum a posterior estimate
of the states $\vec{X}_{0},...,\vec{X}_{T}$ given
$(\vec{Z}_{0},...,\vec{Z}_{T})=(z_{0},...,z_{T})$ is given by the
solution to the problem
\begin{align*} \label{prob:EstConv} \tag{$\mathcal{P}$}
\min_{x_{0},...,x_{T}} & \sum_{t=0}^{T} f_{t}(w_{t}) +
\sum_{t=0}^{T} g_{t}(z_{t} - H_{t} x_{t}) \\
\text{subject to } & x_{t+1} = F_{t} x_{t} + w_{t+1}, \; \;
t=0,...,T-1\\
& x_{0} = w_{0} \\
\end{align*}
Equivalently, one can use an extended formulation, minimizing over
$w_{t}$ and $x_{t}$, or simply minimizing in $w_{t}$.
\end{proposition}
\begin{proof}
In maximum a posteriori estimation, we seek to maximize the density
$$p_{\vec{X}}(\vec{X}_{0},...,\vec{X}_{T} |
(\vec{Z}_{0},...,\vec{Z}_{T}) = (z_{0},...,z_{n})).$$
By Bayes' Theorem
$$p_{\vec{X}}((\vec{X}_{0},...,\vec{X}_{T})=(x_{0},...,x_{T})|
(\vec{Z}_{0},...,\vec{Z}_{T})=(z_{0},...,z_{T})) $$
$$= \frac{p_{\vec{Z}}((\vec{Z}_{0},...,\vec{Z}_{T})=(z_{0},...,z_{T})|
(\vec{X}_{0},...,\vec{X}_{T})=(x_{0},...,x_{T})) 
p_{\vec{X}}(\vec{X}_{0},...,\vec{X}_{T})=(x_{0},...,x_{T}))}{p_{\vec{Z}}((\vec{Z}_{0},...,\vec{Z}_{T})=(z_{0},...,z_{T}))}.$$
By the independence of measurement noise,
$$p_{\vec{Z}}((\vec{Z}_{0},...,\vec{Z}_{T})=(z_{0},...,z_{T})|
(\vec{X}_{0},...,\vec{X}_{T})=(x_{0},...,x_{T})) = \prod_{t=0}^{T}
p_{\vec{V}_{t}}(z_{t} - H_{t} x_{t}).$$
Furthermore, since the process is Markov (by independence of dynamic
noise)
$$p_{\vec{X}}((\vec{X}_{0},..., \vec{X}_{T})=(x_{0},...,x_{T}))
= p_{\vec{X}_0}(x_{0}) \cdot p_{\vec{X}_{1}}(x_{1}|x_{0}) \cdot ...
\cdot  p_{\vec{X}_{T}}(x_{T}|x_{T-1}).$$
Our posterior becomes
$$\frac{\prod_{t=0}^{T} p_{\vec{V}_{t}}(z_{t}-H_{t} x_{t}) \cdot
p_{\vec{X}_{0}}(x_{0}) \cdot
\prod_{t=1}^{T} p_{\vec{X}_{t}}(x_{t} |
x_{t-1})}{p_{\vec{Z}}((\vec{Z}_{0},...,\vec{Z}_{T})=(z_{0},...,z_{T}))}.$$
By the assumptions on the distributions of $\vec{V}_{t}$ and
$\vec{W}_{t}$, this is
$$C(z_{0},...,z_{T}) \cdot \exp \left\{- f_{0}(x_{0}) 
- \sum_{t=0}^{T-1} f_{t+1}(x_{t+1} - F_{t}(x_{t}))
 -\sum_{t=0}^{T} g_{t}(z_{t} - H_{t} x_{t}) \right \}$$
where $C(z_{0},...,z_{n})$ is some term not depending on
$(x_{0},...,x_{T})$. Maximizing this expression in $(x_{0},...,x_{T})$
is then equivalent to minimizing
$$f_{0}(x_{0}) +
\sum_{t=0}^{t-1} f_{t+1}(x_{t+1} - F_{t}(x_{t}))+ \sum_{t=0}^{T} g_{t}(z_{t}-
H_{t} (x_{t})).$$
This gives us the problem in the statement of the proposition. The
different formulations follow because each choice of $(x_{0},,,x_{T})$
generates a unique $(w_{0},...,w_{T})$, according to the dynamics, and
vice versa.
\end{proof}

The extension from the Gaussian noise to log-concave random vectors is signficant.
The fact that the functions $f_{t}$ and $g_{t}$ in \ref{prob:EstConv} can take the
value $\infty$ permits a choice of densities which do
not have full support. Correspondingly, the MAP problem then becomes
one of traditional convex optimization \cite{ConvAnal},
\cite{VarAnal}, where constraints are built in to the objective
function by allowing that function to take infinite values.

The next lemma provides information about the function $f$ used to
define a log-concave density.

\begin{lemma} \label{Lemma}
Assume that $f:\mathbb{R}^{n} \to \overline{\mathbb{R}}$ is a convex function which
defines the density of a random variable $\vec{X} \sim e^{-f(x)}$.
Then
\begin{enumerate}[label=(\alph*)]
\item  If $\cl(f)$ is the lower-semicontinuous hull of $f$, then
$e^{\cl(f)}$ is also a density function for $\vec{X}$.
\item $f$ is proper
\item $dom(f)$ is full-dimensional
\item $f$ is level-bounded, so that the minimum of $f$ over
$\mathbb{R}^{n}$ is attained.
\end{enumerate}
\end{lemma}

\begin{proof}
First we prove (a). Since
a convex function is continuous on the interior of its domain
\cite{ConvAnal}[10.1], the only points where $f$ may
fail to be lower semicontinuous is on the boundary of its domain. The
domain of a convex function is obviously convex, and since the
boundary of a convex set has Lebesgue measure zero
\cite{jarosz2015function}[Lemma 1.8.1], $\cl(f)$ and $f$ are equal
almost everywhere. Hence $e^{\cl(f)}$ is also a density function for
$X$, since it differs from the given density on a set of measure zero.
This results allows us to refer to pointwise values of $f$, by which
we mean the values of the unique lower-semicontinuous extension
$\cl(f)$.

(b) follows from (a) and the fact that $\int e^{-f(x)} dx =1$.
Because an improper lower-semicontinuous convex function can have no finite
values \cite{ConvAnal}[Cor 7.2.1], $f$ must be proper in order for
$e^{-f(x)}$ to integrate to one.

For (c), if $dom(f)$ were not full dimensional then it is a subset of
a proper affine subspace of $\mathbb{R}^{n}$. This set has measure
zero, which violates the condition that $e^{-f(x)}$ integrates to one.

Lastly, we prove (d). In order that $\int e^{-f(x)} dx = 1$, we must
have $f(x) \to \infty$ as $\abs{x} \to \infty$. This means that $f$ is
level bounded, which combined with the fact that we can without loss
of generality take $f$ to be lower-semicontinuous, gives that $f$
attains its minimum \cite{VarAnal}[Thm 1.9]
\end{proof}

To simplify calculations in the results that follow, we will rewrite the
problem \ref{prob:EstConv} in a more compact form. We borrow from Rockafellar \cite{MultStochDual} the
notion of a supervector, which is simply a concatenated vector consisting of a
variable at all time steps. Let $w=(w_{0},...,w_{T})'$,
$z=(z_{0},...,z_{T})'$, $x=(x_{0},...,x_{T})'$ be the supervectors
corresponding MAP estimates of the dynamical noise, measurements, and states,
respectively. Define 
$$f(w)=\sum_{t=0}^{T} f_{t}(w_{t}),$$
$$g(z) = \sum_{t=0}^{T} g_{t}(z_{t}).$$
Note that each of these functions is separable with respect to the
components of their respective supervectors. Hence infimums and
supremums can be performed with respect to each component.

Define 

$$A = \left( \begin{array}{cccc} 1 & 0 & \cdots & 0 \\ -F_{0} & 1 &
\cdots & 0 \\ 0 & \ddots & \ddots & 0 \\ 0 & \cdots & -F_{T-1} & 1
\end{array} \right)$$
so that the dynamical system constraint in \ref{prob:EstConv} can be
represented as 
\begin{equation}\label{dynconst}
Ax - w = 0.
\end{equation}
Similarly, let 
$$H = \left( \begin{array}{cccc} H_{0} & 0 & \cdots & 0 \\ 0 & H_{1} &
\cdots & \\ \vdots & \vdots & \ddots & \vdots \\ 0 & 0 & \cdots &
H_{T} \end{array} \right).$$
Then the measurement constraint can be rewritten in supervector
notation as well, allowing us to rewrite problem \ref{prob:EstConv} 
\begin{align} \label{Supervector}
\min_{x, w} & \quad f(w) + g(z-Hx) \\
\text{s.t. } & Ax - w = 0.\nonumber
\end{align}

We now turn our attention to convex-analytic duality. For
concreteness, we assume that $C \subseteq \mathbb{R}^n$ and $D \subseteq
\mathbb{R}^m$. Recall that a convex problem
$\min_{x \in C} h(x)$ is \emph{dual} to a concave problem $\max_{y
\in D} k(y)$ if there is a convex-concave function $L: \mathbb{R}^{n}
\times \mathbb{R}^m \to \overline{\mathbb{R}}$ such that
$$h(x) = \sup_{y \in D} L(x;y) \;\;\; \text{ and } \;\;\; k(y) =
\inf_{x \in C} L(x;y).$$
This definition, from \cite{ConjDual}, is equivalent to the notion of
duality in which one perturbs constraints in order to
generate a saddle function $L$. Indeed, a perturbation function $F$ can be
generated from the equation 
$$F(x,u) = \sup_{y \in D} \left\{ K(x,y) - u'y \right\}.$$
Of course, this duality framework subsumes the familar Lagrangian
duality, Fenchel duality, and various other duality schemes.
We refer the reader to \cite{ConjDual}
for details and examples, and focus on applying
this theory to the problem at hand.

\begin{theorem}
When $f$  and $g$ are convex, problem
\eqref{prob:EstConv} 
is the primal problem associated with the saddle function 
$$L(w, x; u, y) = f(w) + z'u - g^{*} (u) - u'Hx + y'(Ax - w)$$
on $C:= \R^{(T+1) \times n_{x}} \times \R^{(T+1)\times n_{x}}$,
$D:= \R^{(T+1)\times n_{z}} \times \R^{(T+1) \times n_{x}}$
\end{theorem}
\begin{proof}
To prove that \eqref{prob:EstConv} is the primal problem for the
saddle-function $L$, we will show that 
$$\eqref{prob:EstConv} = \min_{(w,x) \in C} \sup_{(u, y) \in D} L(w,x;u,y)$$
which satisfies the definition in \cite{ConjDual}. For ease of
notation, in what follows we omit the sets over which we take the infimums and
supremums.
By the separability of $L$, 
$$\sup_{u,y} L(w,x; u,y) = f(w) + \sup_{u}\left\{ u'(z-Hx) - g^{*}(u)
\right\}  + \sup_{y}\left\{ y'(Ax-w) \right\}$$
Obviously the right-most supremum is $0$ when $Ax -w=0$ and $\infty$
otherwise. Furthermore, Lemma \ref{Lemma} gives that without loss of
generality $g$ is proper and lsc. Therefore the 
Fenchel-Moreau Theorem \cite[Th 11.1]{VarAnal} gives that
$$\sup_{u} \left\{ u'(z-Hx) - g^{*}(u) \right\} = g^{**}(z-Hx) =
g(z-Hx).$$
Thus
\begin{align*}
&\min_{w,x} \sup_{u,y} L(w,x;u,y) \\
= \min_{w,x} & f(w) + g(z-Hx) \\
\text{s.t. } & Ax-w = 0
\end{align*}
\end{proof}

\begin{theorem} \label{SupervectorDual}
The dual problem associated with $L$ on $\R^{(T+1) \times n_{x}} \times
\R^{(T+1) \times n_{x}}$, $\R^{(T+1) \times n_{z}} \times \R^{(T+1) \times n_{x}}$ is
\begin{align}\label{ContProb}
\sup_{y, u} \;& z'u-f^{*}(y) - g^{*}(u) \\
\text{s.t. } & A'y - H'u = 0 \nonumber
\end{align}
Furthermore, this problem has an equivalent reduced formulation where
the supremum is taken over $u$.
\end{theorem}
\begin{proof}
The dual
problem associated with the triple $L$, $\R^{(T+1) \times n_{x}} \times
\R^{(T+1) \times n_{x}}$, $\R^{(T+1) \times n_{z}} \times \R^{(T+1)
\times n_{x}}$ is
\begin{align*}
&\sup_{u, y} \inf_{w,x} L(w,x;v,y) \\
= &\sup_{u,y} \inf_{w,x} f(w) + z'u - g^{*}(u) - u'Hx + y'(Ax-w) \\
= & \sup_{u,y} \inf_{w,x} f(w) -w'y + x'(A'y-H'u) + z'u - g^{*}(u) \\
= & \sup_{u,y} z'u - g^{*}(u) + \inf_{w} \left\{ f(w) - w'y \right\} +
\inf_{x} \{ x'(A'y-H'u) \}\\
= \sup_{u,y} & z'u - g^{*}(u) -f^{*}(y) \\
\text{s.t. } & A'y-H'u=0 \\
\end{align*}
Lastly, the equivalent reduced formulation follows because the each
$u$ generates a unique $y$ vector according to the constraints.
\end{proof}
Appealing to the separability of $f$ and $g$, and expanding the
matrices $A'$ and $H'$, we have established the following main result.
\begin{theorem} \label{MainTheorem}
The dual problem associated with the estimation problem
\eqref{prob:EstConv} is the optimal control problem
\begin{align}\label{DualConvProbExp} \tag{$\mathcal{D}$}
\sup_{u_{0},...,u_{T}} & \sum_{t=0}^{T} f_{t}^{*}(y_{t}) + g_{t}^{*}(u_{t}) -
z_{t}' u_{t} \\
\text{s.t. } & y_{t} = F_{t}' y_{t+1} + H'_{t} u_{t}, \; \;
t=0,...,T-1 \nonumber \\
& y_{T} = H' u_{T} \nonumber 
\end{align}
\end{theorem}

The next theorem provides a condition, known as a constraint
qualification, for strong duality to hold between the estimation
problem \ref{prob:EstConv} and control \ref{ContProb} problems above.

\begin{theorem}\cite[Theorem 28.2]{ConvAnal} \label{StrongDual}
 Assume that the problem \ref{prob:EstConv} is strictly
feasible, so that there exists a pair $(x,w)$ satisfying
\eqref{dynconst}, $w \in \text{int}(\text{dom}(f))$, $z-Hx \in
\text{int}(\text{dom}(g))$. Then a strong duality relationship exists between the problems
\ref{prob:EstConv} and \ref{DualConvProbExp}. In other words, the supremum in
\ref{DualConvProbExp} equals the optimal
value in \ref{prob:EstConv}. Furthermore, this supremum is attained. 
\end{theorem}
\begin{proof}
This follows directly from the strong duality theorem 
in \cite[Theorem 28.2]{ConvAnal}. Note that the typical formulation of
this constraint qualification requires only a \emph{relative interior}
point, when the domains are considered as subsets of their affine
hulls, instead of an interior point. However, since by Lemma
\ref{Lemma} the domains of $f$ and $g$ are full-dimensional, the
notions are equivalent.
\end{proof}

section{The Piecewise Linear Quadratic Case}

In this section we investigate structural constraints on the functions
$f_{t}$ and $g_{t}$ in the densities of $\vec{W}_{t}$ and $\vec{V}_{t}$ 
that allow us to remove the constraint qualification condition in
\ref{StrongDual}. 

Recall that a function is linear-quadratic if it is polynomial of
degree at most two, so that constant and linear functions are included
in this family. 
\begin{theorem}\label{theorem:PLQ}
 Assume that $f_{t}$ and $g_{t}$ are convex and piecewise 
linear-quadratic. If either \ref{prob:EstConv} or \ref{DualConvProbExp} are
feasible, then strong duality holds between these estimation and
optimal contorl problems, so that
their optimal objective values are equal. Furthermore, both problems
attain their optimal objective values. 
\end{theorem}
\begin{proof} 
If $f_{t}$ and $g_{t}$ are piecewise linear quadratic, then the reformulated
problem \ref{Supervector} is a piecewise linear-quadratic program.
Lemma \ref{Lemma} gives us that each of $f_{t}$ and $g_{t}$ are
proper, and hence each of their conjugates is as well. Combined with the
assumption that one of \ref{prob:EstConv} and \ref{DualConvProbExp} is
feasible, we know that the optimal objective value of this problem is finite. 
Strong duality and the attaining of optimal values then follow directly 
from \cite{VarAnal}[Thm 11.42]
\end{proof}

Theorem \ref{theorem:PLQ} removes the contraint qualification of
Theorem \ref{MainTheorem} by imposing extra structure on the
functions $f_{t}$ and $g_{t}$. By assuming that $f_{t}$ and $g_{t}$
are piecewise linear-quadratic, strong duality becomes automatic. 

When $f_{t}$ and $g_{t}$ are arbitrary convex functions, computing
closed form expressions for the conjugates that occur in the dual
problem may be difficult. 
The conjugate function, Lagrangian, and dual problem
\ref{DualConvProbExp} are especially easy to compute in the
special case that $f_{t}$ and $g_{t}$ are monitoring functions, which
includes many problems of practical interest.

A \emph{monitoring function} is a function $\rho_{U,M} :\mathbb{R}^{n} \to
\overline{\mathbb{R}}$ defined by
$$\rho_{U, M}(x) = \sup_{u \in U}\{ x'u-\frac{1}{2} u' M u \}$$
where $U \subseteq \mathbb{R}^{n}$ is a nonempty polyhedral set and
$M$ is an $n \times n$ positive semidefinite matrix. 

Monitoring functions are flexible tools for modeling penalties. They
are proper, convex, and piecewise linear-quadratic \cite{VarAnal}[Ex
11.18], and can be used to model a variety of linear and quadratic
penalties in addition to polyhedral constraints. A probabilistic interpretation
of the use of monitoring functions in robust smoothing problems can be found
in \cite{aravkin}. The authors detail the construction of many
commonly used penalties in robust optimization, and provide remarks
for constructing others.

The incentive for considering the case that $f_{t}$ and $g_{t}$ are
monitoring functions is two-fold. The first is that, though this is
requires a further restriction on the form of $f_{t}$ and $g_{t}$, the
MAP problem Gaussian noise  is contained in this
case. The second is that the framework provides enough structure to
make the computation of the conjugates and the dual control problem
\ref{DualConvProbExp} straight-forward.

We'll be aided in this by the following lemma.
\begin{lemma}
If $\rho_{U,M}$ is a monitoring function, then the conjugate
$\rho^{*}_{U, M}$ is given by
$$\rho^{*}_{U,M}(y) = \left\{ \begin{array}{cc} \frac{1}{2} y' M y &
\text{ when } y \in U \\ \infty & \text{ otherwise} \end{array}
\right.$$
\end{lemma}
\begin{proof}
\begin{align} \label{conj}
\rho^{*}_{U, M}(y) &= \sup_{x \in \mathbb{R}^{n}} \left\{ x' y
 - \rho_{U, M}(x) \right\} \nonumber \\
&= \sup_{x \in \mathbb{R}^{n}} \left\{ x'y - \sup_{u
\in U} \left\{ x'u - \frac{1}{2} u' M u \right\} \right\} \nonumber \\
&= \sup_{x \in \mathbb{R}^{n}} \inf_{u \in U} \left\{ x'(y-u) +
\frac{1}{2} u' M u \right\} 
\end{align}
This can be interpreted as the dual to the problem 
\begin{align*} 
\inf_{u \in U} & \frac{1}{2} u' M u  \\
\text{s.t. }&  u=y 
\end{align*}
In the same spirit as \ref{theorem:PLQ}, we apply \cite{VarAnal}[11.42]
to see that, when $U$ is nonempty, \ref{conj} equals
\begin{align*}
\inf_{u \in U} \sup_{x \in \mathbb{R}^{n}} & \left\{ x'(y-u) + \frac{1}{2} u' M u
\right\} \\
&= \left\{ \begin{array}{cc} \frac{1}{2} y' M y &
\text{ when } y \in U \\ \infty & \text{ otherwise} \end{array}
\right.
\end{align*}
\end{proof}

We've arrived at a precise formulation of the dual control
problem \ref{DualConvProbExp}.
\begin{corollary}\label{theorem:DualMonitor}
If $\vec{W}_{t}$ and $\vec{V}_{t}$ have a PLQ density
$$\vec{W}_{t} \propto e^{-\rho_{W_{t}, M_{t}}} \;\;\; \vec{V}_{t} \propto
e^{-\rho_{V_{t}, N_{t}}}$$
with $M_{t}, N_{t} \succeq 0$ and and $W_{t}$, $V_{t}$ nonempty and
polyhedral,
then the MAP problem \ref{prob:EstConv} is dual to the control problem
\begin{align}\label{DualMonitor}
\sup_{u_{0},...,u_{T}} & \sum_{t=0}^{T} \frac{1}{2} y_{t}' M_{t} y_{t} +
\frac{1}{2} u_{t}' N_{t} u_{t} - z_{t}' u_{t} \\
\text{s.t. } & y_{t} = F_{t}' y_{t+1} + H'_{t} u_{t}, \; \;
t=0,...,T-1 \nonumber \\
& y_{T} = H' u_{T} \nonumber \\
& y_{t} \in W_{t}, \;\;\; t=0,...,T \\
& u_{t} \in V_{t}, \;\;\; t=0,...,T 
\end{align}
Optimal values are attained and strong duality holds when either
problem is feasible.
\end{corollary}
\begin{proof}
Follows directly from \ref{MainTheorem} and \ref{theorem:PLQ}.
\end{proof}
As an application of Theorem \ref{theorem:DualMonitor}, we verify the strong
duality between estimation and control in the linear-quadratic
Gaussian setting. By taking $\vec{W}_{t} \sim \mathcal{N}(0,Q_{t})$ and $\vec{V}_{t} \sim
\mathcal{N}(0,R_{t})$, we can represent the density functions as in
theorem \ref{theorem:DualMonitor} by taking
$$W_{t}=\mathbb{R}^{n_{x}}, \; M_{t} = Q_{t}, \;\;\;\; V_{t} =
\mathbb{R}^{n_{z}}, N_{t} = R_{t}$$ 
Because $\rho_{t, \mathbb{R}^{n_{x}}, Q_{t}}$ and $\rho_{t,
\mathbb{R}^{n_{z}}, R_{t}}$ have domains $\mathbb{R}^{n_{x}}$ and
$\mathbb{R}_{n_{z}}$, respectively, the MAP problem is feasible for
any measurements $(z_{0},...,z_{T})$. Strong duality then follows
directly from Theorem \ref{theorem:DualMonitor}.

Moreover, by taking $\vec{W}_{t} \sim \mathcal{N}(0, Q_{t})$ and
$\vec{V}_{t} \equiv 0$, we recover the duality of the two problems
considered by Kalman from the introduction. The time reversal is in fact an
artifact of taking the convex analytic dual, though recovery of the
Riccati-covariance propagation equivalence requires considering each problem as sequential,
which is not the appoach that we've taken here.

\section{Applications: Reconstructing an Estimator from Optimal Controls}
In this section we apply the results of the previous sections to
construct the solution to an optimal estimation problem from the solution to its dual problem of
optimal control. We focus on a nonsmooth problem of practical
interest. First, we formulate an
estimation problem where the density of the measurement noise is generated via
log-concave maximum likelihood estimation from a sample of measurement noise. We
then use the result \ref{MainTheorem} to construct a corresponding
dual control problem. Finally, we use the solution to this control problem
to construct an optimal estimator for the original problem.

The set up for this problem motivated by the following scenario. A
practioner aims to estimate the current and previous states of a
dynamical system from a set of
noisy measurements. Through calibration of a sensor or observation, the practioner
has the ability to generate sample data for the measurement noise. How
can one use this data to formulate an estimation problem which reflects
the tendencies of the sensor? We would like to capitalize on the
ability to generate measurement noise, instead of defaulting to a
Gaussian noise assumption.

The following theorem, due to \cite{rufibach} and \cite{pal}, as well as the
computational results in \cite{CRANPackage}, allow us to form a
nonparametric density estimate of a log-concave density based on a set
of sample data.
\begin{theorem}
If $X_1,...,X_n$ are i.i.d. observations from a univariate log-concave
density, then the nonparamtric MLE exists, is unique, and is of the
form $\hat{\phi}_{n} = \exp \hat{f_{n}}$. The function $\hat{f}_{n}$
is
piecewise linear on $[X_{(1)}, X_{(n)}]$, with the set of knots
contained in $\{ X_{1},...,X_{n} \}$. Outside of $[X_{(1)}, X_{(n)}]$,
$\hat{f}$ takes the value $\infty$.
\end{theorem}

For the generation of our example problem, we use 10 time steps
and a two dimensional state space, motivated by components representing position and
velocity. The dynamics matrices $F_{t}$ corresponds to the physical
dynamics that would occur in such a situation. We produce a ``true''
sequence of states to be estimated by generating dynamical system
noise according to a $\mathcal{N}(0,I)$ distribution.
We take the measurement operators $H_{t}$ be the sum of the components.
Lastly, we construct measurements from a sample of
Laplace$(0,1)$ measurement noise.  

For the formulation of the MAP problem, we assume that the dynamical
system noise was generated according to $\mathcal{N}(0,I)$
distribution. For the measurement noise, we construct a log-concave
MLE estimator $e^{-\hat{g}}$ of the density from a sample of size 100 generated from
a Laplace(0,1) distribution. $\hat{g}$ and its convex conjugate $\hat{g}^{*}$
are illustrated in figure \ref{figure1}.

\begin{figure}[H]
\centering
\begin{subfigure}{.5\textwidth}
  \centering
  \includegraphics[width=\linewidth]{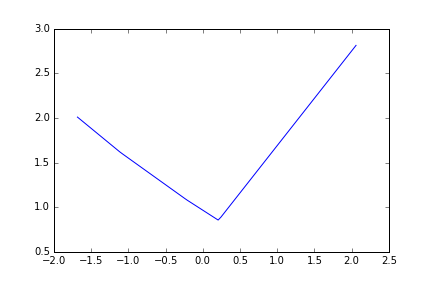}
  \caption{$\hat{g}$}
\end{subfigure}%
\begin{subfigure}{.5\textwidth}
  \centering
  \includegraphics[width=\linewidth]{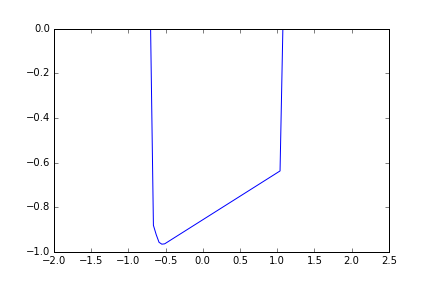}
  \caption{$\hat{g}^{*}$}
\end{subfigure}%
\caption{$\hat{g}$ and $\hat{g}^{*}$, where the MLE density is
$e^{-g(x)}$}
\label{figure1}
\end{figure}

The MAP problem \ref{prob:EstConv} is then
\begin{align*} \label{P:example} \tag{$\mathcal{P}_{ex}$}
\min_{x_{0},...,x_{10}} & \sum_{t=0}^{10}  \frac{1}{2} \norm{w_{t}}^{2} + \hat{g}(z_{t}- (\begin{array}{cc} 1 &
1 \end{array}) \cdot x_{t}) \\
\text{s.t. } & x_{t+1} = \left( \begin{array}{cc} 1 & 1 \\ 0 & 1
\end{array} \right) x_{t} + w_{t+1}, \;\;\; t=0,...,T-1\\
& x_{0} = w_{0} 
\end{align*}

According to \ref{MainTheorem} this problem has as its dual the
control problem
\begin{align*} \label{D:example} \tag{$\mathcal{D}_{ex}$}
\max_{u_{0},...,u_{10}} & \sum_{t=0}^{10}  \frac{1}{2} \norm{y_{t}}^{2} + \hat{g}^{*}(u_{t}) - u_{t}' z_{t} \\
\text{s.t. } & y_{t} = \left( \begin{array}{cc} 1 & 0 \\ 1 & 1
\end{array} \right) y_{t+1} + \left( \begin{array}{c} 1 \\ 1 \end{array}
\right) u_{t}, \;\;\; t=0,...,T-1 \\
& y_{T} = \left( \begin{array}{c} 1 \\ 1 \end{array} \right) u_{T} 
\end{align*}

Because our dynamical system noise in \ref{P:example} has full
support, Theorem \ref{StrongDual} guarantees that strong duality holds
between the problems, and that the dual control problem attains its
solution. Assume that we have solved this control problem and have a
corresponding optimal control $(u^{*}, y^{*})$. Since an optimal
estimate $(w^{*}, x^{*})$ gives a saddle point $(w^{*}, x^{*}; u^{*},
y^{*})$ to the Lagrangian $L$ in
Theorem \ref{SupervectorDual}, it follows from the
proof of this theorem that $w^{*}$ minimizes $f(w) - w'y^{*}$. In our
problem $f$, when $f(w) = \frac{1}{2} \norm{w}^2$, which yields the
relationship $w^{*} = y^{*}$. This is similar to the
relationship between primal and dual solutions in the Fenchel Duality
framework. See \cite{Bertsekas}[Prop. 5.3.8] for further details. Note that this
allows us to reconstruct a primal solution from a dual
solution and vice versa. In particular, the relationship between
$y^{*}$ and $w^{*}$ is
linear when the dynamical system noise is assumed to
be Gaussian.

Solving \ref{D:example} to optimality \footnote{Experiments
available from first author's website:
math.ucdavis.edu/{\textasciitilde}rbassett} gives
$y^{*}$, from which we generate $w^{*}$ and then an optimal estimate
$x^{*}$. Figure \ref{figure2} contains a plot of the estimated and
true state.

\begin{figure}
\centering
  \includegraphics[width=\linewidth]{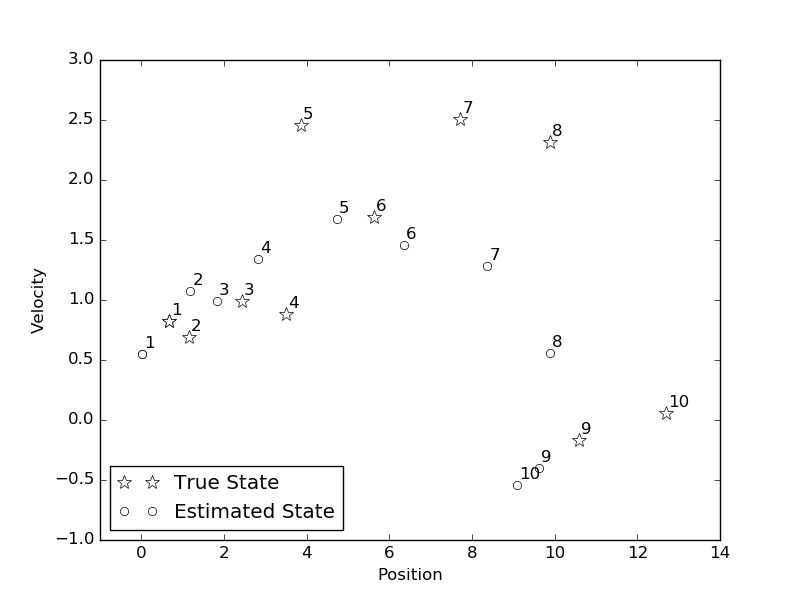}
\caption{Dual Reconstruction of State Estimate}
\label{figure2}
\end{figure}

Though we have demonstrated a convenient technique to generate
solutions for an estimation problem from the solution to its dual
control problem, in this
example we have no
reason to believe that solving \ref{D:example} is any
easier than solving the original problem \ref{P:example}.
Nevertheless, the results in this and previous sections provide motivation for further
investigation into applying control algorithms to solve estimation
problems and vice versa.

\printbibliography
\end{document}